\documentclass[12pt, oneside]{amsart}
\usepackage{amsthm}
\usepackage{fullpage}
\usepackage[margin=1in]{geometry}
\usepackage{amsmath}
\usepackage{mathrsfs}
\usepackage{amssymb}
\usepackage{subfigure}
\usepackage{verbatim}
\usepackage{caption}
\usepackage{color}
\usepackage[dvipsnames]{xcolor}
\usepackage{enumerate}
\usepackage{tikz}
\usetikzlibrary{positioning}

\usepackage{array}
\usepackage{graphicx}
\usepackage{subfigure}
\newcommand{\Cov}[1]{\mathscr{#1}}
\def\HH{\Cov{H}}
\newcommand{\bc}{{\bf c}}

\newtheorem{theorem}{Theorem}[section]
\newtheorem{proposition}[theorem]{Proposition}

\newtheorem{observation}[theorem]{Observation}
\newtheorem{corollary}[theorem]{Corollary}
\newtheorem{claimy}[theorem]{Claim}
\newtheorem{lemma}[theorem]{Lemma}

\theoremstyle{definition}
\newtheorem{definition}{Definition}
\newtheorem{claimx}{Claim}
\newenvironment{claim}
  {\pushQED{\qed}\claimx}
  {\popQED\endclaimx}
\begin{document}

\author{Alexandr Kostochka}
\address{%
Department of Mathematics, University of Illinois at Urbana-Champaign, Urbana, IL 61801, USA, and Sobolev Institute of Mathematics, Novosibirsk 630090, Russia.}
\email{kostochk@illinois.edu.}

\author{Jingwei Xu}
\address{%
Department of Mathematics, University of Illinois at Urbana-Champaign, Urbana, IL 61801, USA.}
\email{jx6@illinois.edu.}

\thanks{A.K.~is partially supported by NSF grant   DMS-2153507 and NSF RTG grant DMS-1937241.}
\thanks{J.X.~is partially supported   by NSF grant DMS-2153507. }

\title{Sparse critical graphs for defective DP-colorings}

\maketitle

\begin{abstract}
An interesting generalization of list coloring is so called
DP-coloring (named after Dvo\v r\' ak and Postle). We study $(i,j)$-defective DP-colorings of simple graphs. Define $g_{DP}(i,j,n)$ to be the minimum number of edges in an $n$-vertex DP-$(i,j)$-critical graph.  We prove  sharp bounds on $g_{DP}(i,j,n)$ for  $i=1,2$ and  $j\geq 2i$  for infinitely many $n$.
\end{abstract}

\section{Introduction}

A \emph{$(d_1, \dots, d_k)$-defective coloring} (or simply  {\em $(d_1, \dots, d_k)$-coloring}) of a graph $G$ is a partition of $V(G)$ into sets $V_1,V_2,\dots,V_k$
such that for every $i\in[k]$, each vertex in $V_i$ has at most $d_i$ neighbors in $V_i$. In particular, a $(0,0,\ldots,0)$-defective coloring is 
the ordinary proper $k$-coloring.
Defective colorings have been studied in a number of papers,
  see e.g.~\cite{Ar1,CCW1,EKKOS,EMOP,HW18,KYu,LL66,OOW,VW}.

 In this paper we consider colorings  with 2 colors.
 If $(i,j)\neq (0,0)$, then the problem
 to decide whether
 a graph $G$ has an $(i,j)$-coloring  is  NP-complete.  
In view of this, a direction of study is to find how sparse can be graphs with no $(i,j)$-coloring for given $i$ and $j$; 
see e.g.~\cite{BIMOR10,BIMR11,BIMR12,BK11,BK14,BKY13,KKZ14,KKZ15}. A natural measure of  ``sparsity" of a graph 
 is the {\em maximum average degree},
$mad(G)=\max_{G'\subseteq G}\frac{2|E(G')|}{|V(G')|}$. In such considerations, an important notion is
that of
 $(i,j)$-{\em critical graphs}, that is,
the graphs that do not have $(i,j)$-coloring but every proper subgraph of which has such a coloring. Let
$f(i,j,n)$  denote the minimum number of edges in an $(i,j)$-critical $n$-vertex graph. Observe that  for odd $n\geq 3$ we have
 $f(0,0,n)=n$. Indeed, for odd $n$
the $n$-cycle is not bipartite, but every $n$-vertex graph with fewer than $n$ edges has a vertex of degree at most $1$ and thus cannot be $(0,0)$-critical. The reader can find
interesting bounds on $f(i,j,n)$ in the papers cited above. 

\bigskip

A \emph{$k$-list} for a graph $G$ is a function $L: V(G)\rightarrow \mathcal{P}(\mathbb{N})$ such that $|L(v)|=k$ for every $v\in V(G)$.
   A {\em $d$-defective  $L$-coloring}
of $G$ is a function $\varphi : V(G) \to \bigcup_{v\in V(G)} L(v)$ such that $\varphi(v)\in L(v)$ for every $v\in V(G)$ and every vertex has at most
$d$ neighbors of the same color. If $G$ has a $d$-defective  $L$-coloring from every $k$-list assignment $L$, then $G$ is called
{\em $d$-defective $k$-choosable}.
These   notions
 were introduced in~\cite{EH1,S1999} and studied in~\cite{S2000,Woodall1,HS06,HW18}.
A direction of study is showing that ``sparse" graphs are $d$-defective $k$-choosable.
The best known bounds on maximum average degree that guarantee that a
graph is $d$-defective $2$-choosable are due to Havet and Sereni~\cite{HS06}: 

\bigskip
\noindent{\bf Theorem A} (\cite{HS06}){\bf.} {\em For every $d\geq 0$, if $mad(G)<\frac{4d+4}{d+2}$, then $G$ is $d$-defective $2$-choosable.
On the other hand, for every $\epsilon>0$, there is a graph $G_\epsilon$ with $mad(G_\epsilon)<4+\epsilon-\frac{ 2d+4}{d^2+2d+2}$ that is
not $(d,d)$-colorable.}

\bigskip


 Dvo\v r\'ak and Postle \cite{DP18} introduced and studied  so called {\em DP-coloring} which is more general than list coloring.  Bernshteyn, Kostochka and Pron~\cite{BKP}  extended this notion to multigraphs.


\begin{definition}\label{defn:cover}
		Let $G$ be a multigraph. A \emph{cover} of $G$ is a pair $\HH=(L, H)$, consisting of a graph $H$ (called the \emph{cover graph} of $G$) and a function $L \colon V(G) \to 2^{V(H)}$, satisfying the following requirements:
		\begin{enumerate}
			\item the family of sets $\{L(u) \,:\,u \in V(G)\}$ forms a partition of $V(H)$;
			\item for every $u \in V(G)$, the graph $H[L(u)]$ is complete;
			\item if $E(L(u), L(v)) \neq \varnothing$, then either $u = v$ or $uv \in E(G)$;
				\item \label{item:matching} if the multiplicity of an edge $uv \in E(G)$ is $k$, then $H[L(u), L(v)]$ is the
union  of at most $k$  matchings connecting $L(u)$ with $L(v)$.
		\end{enumerate}
		A cover $(L, H)$ of $G$ is \emph{$k$-fold} if $|L(u)| = k$ for every $u \in V(G)$.
	\end{definition}

Throughout this paper, we consider only $2$-fold covers. And we call the vertices in the cover graph by ``nodes'', in order to distinguish them from the vertices in the original graph.
				
	 \begin{definition}
		Let $G$ be a multigraph and $\HH=(L, H)$ be a cover of $G$. An \emph{$\HH$-map} is
	an injection $\varphi: V(G)\rightarrow V(H)$, such that $\varphi(v)\in L(v)$ for every $v\in V(G)$. The subgraph of $H$ induced by $\varphi(V(G))$ is called the \emph{$\varphi$-induced cover graph}, denoted by $H_{\varphi}$.
	 \end{definition}
	
   \begin{definition}\label{ijcolor} Let $\HH=(L, H)$ be a cover of $G$. For $u \in V(G)$, let $L(u)=\{p(u), r(u)\}$, where $p(u)$ and $r(u)$ are called the \emph{poor} and the \emph{rich} nodes, respectively.	Given $i, j\geq 0$ and $i\leq j$, an $\HH$-map $\varphi$ is an \emph{$(i, j)$-defective-$\HH$-coloring of $G$} if the degree of every poor node in $H_\varphi$ is at most $i$, and the degree of every rich node in $H_\varphi$ is at most $j$.
   \end{definition}

	\begin{definition}\label{D-S-1} A multigraph $G$ is \emph{$(i,j)$-defective-DP-colorable} if for every $2$-fold cover $\HH=(L,H)$ of $G$, there exists an $(i, j)$-defective-$\HH$-coloring. We say $G$ is \emph{$(i,j)$-defective-DP-critical}, if $G$ is not $(i,j)$-defective-DP-colorable, but every proper subgraph of $G$ is.
\end{definition}

   For brevity, in the rest of the paper, we call an $(i, j)$-defective-$\HH$-coloring simply by an \emph{$(i,j,\HH$)-coloring} (or `$\HH$-coloring', if
   $i$ and $j$ are clear from the context). Similarly, instead of ``$(i, j)$-defective-DP-colorable" and ``$(i,j)$-defective-DP-critical'' we will say  ``$(i,j)$-colorable" and ``$(i,j)$-critical".

We say a $2$-fold cover $\HH=(L,H)$ of a simple graph $G$ is {\em full} if for every edge $uv\in E(G)$, the matching in $H$ corresponding to $uv$ is perfect.
To show that $G$ is $(i,j)$-colorable, it suffices to show that $G$ admits an $(i,j,\HH)$-coloring for all full $2$-fold cover $\HH$ of $G$. Hence below we consider only full covers.

  Denote the minimum number of edges in an $n$-vertex \emph{$(i,j)$-critical multigraph} by $f_{DP}(i,j,n)$, and
 the minimum number of edges in an $n$-vertex \emph{$(i,j)$-critical simple graph} by $g_{DP}(i,j,n)$. By definition,
 $  f_{DP}(i,j,n)\leq g_{DP}(i,j,n)$.
  Jing, Kostochka, Ma, Sittitrai and Xu~\cite{JKMSX} proved lower bounds on $f_{DP}(i,j,n)$  that are exact for  infinitely many $n$
  for every choice of $i\leq j$.
\bigskip

\noindent{\bf Theorem B} (\cite{JKMSX}){\bf.} {\em
\begin{enumerate}
        \item If $i = 0$ and $j\geq 1$, then $f_{DP}(0,j,n) \geq n+j$. This is sharp for every $j\geq 1$       and every $n\geq 2j+2$.
        \item If $i\geq1$ and $j\geq 2i+1$, then $f_{DP}(i,j,n)\geq \frac{(2i+1)n-(2i-j)}{i+1}$. This is sharp for each such pair $(i,j)$ for infinitely many $n$.
        \item If $i\geq1$ and $i+2\leq j\leq 2i$, then $f_{DP}(i,j,n)\geq \frac{2jn+2}{j+1}$. This is sharp for each such pair $(i,j)$ for infinitely many $n$.
        \item If $i\geq1$, then $f_{DP}(i,i+1,n)\geq \frac{(2i^2+4i+1)n+1}{i^2+3i+1}$. This is sharp for each  $i\geq 1$ for infinitely many $n$.
        \item If $i\geq1$,  then $f_{DP}(i,i,n)\geq \frac{(2i+2)n}{i+2}$. This is sharp for each $i\geq 1$ for infinitely many $n$.
    \end{enumerate}
The bound in Part (1) is also sharp for simple graphs. }

\bigskip
For $i>0$ we know little if  the bounds of Theorem~B are sharp on simple graphs. In fact, we think that $g_{DP}(i,j,n)>f_{DP}(i,j,n)$ for $i>0$.
It follows from~\cite{KX} that $g_{DP}(1,1,n)>f_{DP}(1,1,n)$ and $g_{DP}(2,2,n)>f_{DP}(2,2,n)$.
Recently  Jing, Kostochka, Ma and Xu~\cite{JKMX}  showed that when $i\geq 3$ and $j\geq 2i+1$, the exact lower bound for $g_{DP}(i,j,n)$ differs from the bound of Theorem ~B(2) but only by 1.
\vspace{2mm}
\\
\noindent{\bf Theorem C} (\cite{JKMX}){\bf.} {\em
Let $i \geq 3$, $j \geq 2i+1$ be positive integers, and let $G$ be an $(i,j)$-critical simple graph.
Then $$g_{DP}(i,j,n) \geq \frac{(2i+1)n+j-i+1}{i+1}.$$
This is sharp for each such pair $(i,j)$ for infinitely many $n$.
}

The goal of this paper is to extend Theorem C to $i = 1,2$ and $j\geq 2i$, and to show that our bound is exact for infinitely many $n$ for each such pair $(i,j)$. 

\section{Results}

The main result of this paper is  the following.

\begin{theorem}\label{MT-f}
Let $i = 1,2$, $j \geq 2i$ be positive integers, and let $G$ be an $(i,j)$-critical simple graph.
Then $$g_{DP}(i,j,n) \geq \frac{(2i+1)n+j-i+1}{i+1}.$$
This is sharp for each such pair $(i,j)$ for infinitely many $n$.
\end{theorem}

Comparing Theorem~\ref{MT-f} with Theorem B(3) we see that not only
 $g_{DP}(1,2,n)>f_{DP}(1,2,n)$ and $g_{DP}(2,4,n)>f_{DP}(2,4,n)$, but also the asymptotics when $n\to\infty$ are different. 

Since every non-$(i,j)$-colorable graph contains an $(i, j)$-critical subgraph, Theorem~\ref{MT-f} yields the following.

\begin{corollary}
Let $G$ be a simple graph. If $i=1,2$ and $j\geq 2i$ and for every subgraph $H$ of $G$, $|E(H)| \leq \frac{(2i+1)|V(H)|+j-i}{i+1},$ then $G$ is $(i,j)$-colorable. This is sharp.
\end{corollary}

In the next section we introduce a more general framework to prove the lower bound of Theorem~\ref{MT-f}.
In Section 4 we prove some useful lemmas that apply to all pairs of $i=1,2, j\geq 2i$. In Sections 5 and 6, we prove  Theorem~\ref{MT-f} for $i=1, j\geq 2$,  and for $i = 2, j \geq 4$, respectively. In the last section, we present constructions showing that our bounds are sharp for each pair of $i=1,2, j\geq 2i$ for infinitely many $n$.

\section{Notation and a more general framework}

For induction purposes, it will be useful to prove a more general result.
 Instead of $(i,j)$-colorings of a cover $(L,H)$ of a graph $G$, we will consider $(L,H)$-maps $\varphi$ with variable restrictions
 on the  degrees of the vertices in $H_\varphi$.

\begin{definition}
\label{M-D-1}
A \emph{capacity function} on $G$ is a map ${\bf c}: V(G)\rightarrow \{-1,0,\dots,i \}\times \{-1,0,\dots,j\}$. For $u\in V(G)$, denote ${\bf c}(u)$ by $({\bf c}_1(u), {\bf c}_2(u))$.
We call such pair $(G,{\bf c})$ a \emph{weighted pair}.
\end{definition}

 Below, let $(G,{\bf c})$ be a weighted pair, and $\HH=(L,H)$ be a cover of $G$.
For a subgraph $G'$ of $G$, let $\HH_{G'} = (L_{G'}, H_{G'})$ denote the subcover \emph{induced} by $G'$, i.e.,

(1) $L_{G'} = L|_{V(G')}$, where `$f|_S$' is the restriction of function $f$ to subdomain $S$; 

(2) $V(H_{G'}) = L(V(G'))$ and $L_{G'}(v) = L(v)$ for every $v\in V(G')$;

 (3) $H_{G'}[L(u)\cup L(v)] = H[L(u)\cup L(v)]$ for every $uv\in E(G')$, and for $x,y$ such that $xy\notin E(G')$, there is no edge between $L_{G'}(x)$ and $L_{G'}(y)$.

For a subset $S$ of $V(G)$, let $\HH_S = (L_S, H_S)$ denote the subcover induced by $G[S]$. If a capacity function is the restriction of ${\bf c}$ to some $S\subseteq V(G)$, we denote this capacity function by ${\bf c}$ instead of ${\bf c}|_S$, for simplicity.
Let $N_G(S)$, or $N(S)$ when clear from the context, denote the vertices in $G-S$ having a neighbor in $S$. When $S = \{v\}$ for some single vertex $v$, we write $N(v)$ instead of $N(\{x\})$ for simplicity.

For two vertices/nodes $x,y$, we use \emph{$x\sim y$} to indicate  that $x$ is adjacent to $y$, and $x\nsim y$ to indicate that $x$ is not adjacent to $y$.

\begin{definition}
\label{ijc}
A \emph{$({\bf c},\HH$)-coloring of $G$} is an $\HH$-map $\varphi$ such that for each $u\in V(G)$, the degree of $p(u)$ in $H_{\varphi}$ is at most ${\bf c}_1(u)$, and that of $r(u)$ is at most ${\bf c}_2(u)$.
We call ${\bf c}_1(u)$ the \emph{capacity} of $p(u)$  and ${\bf c}_2(u)$ the \emph{capacity} of $r(u)$.
If the capacity of some $v$ in $V(H)$ is $-1$, then $v$ is not allowed in the image of any $({\bf c},\HH$)-coloring of $G$. If for every cover $\HH$ of $G$, there is a $({\bf c},\HH$)-coloring, we say that $G$ is {\em ${\bf c}$-colorable}.
\end{definition}

If ${\bf c}(v) = (i,j)$ for all $v\in V(G)$, then any $( {\bf c},\HH$)-coloring of $G$ is an  $(i,j,\HH$)-coloring in the  sense of Definition~\ref{ijcolor}. So, Definition~\ref{ijc} is
a refinement of Definition~\ref{ijcolor}. Similarly, we say that $G$ is ${\bf c}$-{\em critical} if $G$ is not ${\bf c}$-colorable, but every proper subgraph of $G$ is.
For every node $x$ in the cover graph, we slightly abuse the notation of ${\bf c}$ and denote the capacity of $x$ by ${\bf c}(x)$.
When $\HH$ is clear from the context, say it is equal to $\HH$ or some induced subcover, we drop the corresponding cover notation and say $G$ has a $\bc$-coloring for some capacity function $\bc$, for simplicity.


\begin{definition}\label{M-D-2} For a vertex $u \in V(G)$, the \emph{$(i,j, {\bf c})$-potential} of $u$ is $$\rho_{{\bf c}}(u):= i-j+1+{\bf c}_1(u)+{\bf c}_2(u).$$
The $(i,j, {\bf c})$-potential of a subgraph $G'$ of $G$ is
\begin{equation}\label{eq-1}
\rho_{G,{\bf c}}(G') := \sum\nolimits_{u \in V(G')}\rho_{{\bf c}}(u)-(i+1)|E(G')|.
\end{equation}
For a subset $S\subseteq V(G)$, the $(i,j, {\bf c})$-potential of $S$, $\rho_{G,{\bf c}}(S)$, is the $(i,j, {\bf c})$-potential of $G[S]$.
The $(i,j, {\bf c})$-potential of $(G,{\bf c})$ is defined by $\rho(G,{\bf c}) := \rho_{G,{\bf c}}(G)$.
\end{definition}
When clear from context, we call the $(i,j, {\bf c})$-potential simply by potential.

If follows from the definition (see, e.g. Lemma 3.1 in~\cite{JKMX})
 that the potential function is submodular:
For all $A,B\subseteq V(G)$,
\begin{equation}\label{eq:submod}
    \rho_{G,{\bf c}}(A)+\rho_{G,{\bf c}}(B)=\rho_{G,{\bf c}}(A\cup B)+\rho_{G,{\bf c}}(A\cap B)+(i+1)|E(A\setminus B,B\setminus A)|.
\end{equation}

The main lower bound in  this paper is the following.


\begin{theorem}\label{MT-F}
Let $i =1,2$, $j \geq 2i$ be positive integers, and $(G,{\bf c})$ be a weighted pair such that $G$ is  ${\bf c}$-critical. Then $\rho(G,{\bf c}) \leq i-j-1$.
\end{theorem}
To deduce the lower bound in Theorem \ref{MT-f}, simply set ${\bf c}(v)=  (i,j)$ for every $v\in V(G)$. We will prove the lower bound of Theorem~\ref{MT-F} in the next three sections and present a sharpness construction in Section 7.




\section{Basic lemmas}
Let $1\leq i\leq 2$  and $j\geq 2i$.
 Suppose there exists a ${\bf c}$-critical graph $G$ with
$\rho(G, {\bf c}) \geq i-j$.  Choose such $(G,{\bf c})$ with $|V(G)|+|E(G)|$ minimum. We say that $G'$ is \emph{smaller} than $G$ if $|V(G')|+|E(G')|< |V(G)|+|E(G)|$. Let $\HH=(L,H)$ be an arbitrary cover of $G$.
In this section we show some properties of smallest counterexamples  $(G,{\bf c})$.

\begin{lemma}\label{LM-M-1} Let $S$ be a proper subset of $V(G)$. If $\rho_{G,{\bf c}}(S) \leq i-j$, then $S=\{x\}$ for some $x\in V(G)$ with $\rho_{{\bf c}}(x)=i-j$.
\end{lemma}
\begin{proof}
Suppose the lemma fails. Let $S$ be a maximal proper subset of $V(G)$ such that $\rho_{G,{\bf c}}(S)\leq i-j$ and $|S|\geq 2$.
If $|N(v)\cap S|\geq 2$ for some $v\in V(G)\setminus S$, then
\[
\rho_{G,{\bf c}}(S\cup\{v\})\leq  i-j+2i+1-2(i+1) = i-j-1.
\]
If $S\cup \{v\}\neq V(G)$, this contradicts the maximality of S, otherwise this contradicts the choice of $G$.  Thus
\begin{equation}\label{new1}
\mbox{\em $|N(v)\cap S|\leq 1$ for every $v\in V(G)\setminus S$.}
\end{equation}
Since $G$ is ${\bf c}$-critical,  $G[S]$ admits an $({\bf c},\HH_S)$-coloring $\varphi$.

Construct $G'$ from $G-S$ by adding a new vertex $v^*$ adjacent to every $u\in V(G)-S$ that was adjacent to a vertex in $S$.
Define a capacity function ${\bf \bc'}$ by letting
 ${\bf \bc'}(v^*) = (-1,0)$ and ${\bf \bc'}(u) = {\bf c}(u)$ for $u\in V(G'-v^*)$.

 By~\eqref{new1}, $G'$ is simple. Suppose $\rho_{G',{\bf \bc'}}(A)\leq i-j-1$ for some $A\subseteq V(G')$. Since $G'-v^*\subseteq G$ and
 ${\bf \bc'}(u) = {\bf c}(u)$ for $u\in V(G'-v^*)$, $v^*\in A$, otherwise it contradicts the choice of $G$. Then using~\eqref{eq:submod} and $\rho_{G,{\bf c}}(S)\leq i-j=\rho_{G',{\bf \bc'}}(v^*)$,
 $$\rho_{G,{\bf c}}(S\cup (A-v^*))=\rho_{G,{\bf c}}(S)+\rho_{G,{\bf c}}(A-v^*)-(i+1)|E_G(S,A-v^*)|
 $$
 $$\leq \rho_{G',{\bf \bc'}}(v^*)+\rho_{G',{\bf \bc'}}(A-v^*)-(i+1)|E_{G'}(v^*,A-v^*)|=\rho_{G',{\bf \bc'}}(A)\leq i-j-1.
 $$
 Again, this contradicts either the maximality of $S$ or the choice of $G$. This yields
 \begin{equation}\label{new2}
\rho(G',{\bf \bc'})\geq i-j.
\end{equation}

For every $x\in S$ and $y\in N(x)\setminus S$, denote the neighbor of $\varphi(x)$ in $L(y)$ by $y_{\varphi}$. Let $\HH' = (L',H')$ be the cover of $G'$ defined as follows:

1) $L'(v^*) = \{p(v^*), r(v^*) \}$, and $L'(u) = L(u)$ for every $u\in V(G)\setminus S$;

2) $y_{\varphi}\sim r(v^*)$ for every $y\in N(S)$,  and    $H'[\{u,w\}] = H[\{u,w\}]$ for every $u,w \in V(G'-v^*)$.

By~\eqref{new2} and the minimality of $G$, $G'$ has a $({\bf \bc'}, \HH')$-coloring $\psi$. Since ${\bf \bc'}(v^*) = (-1,0)$,
\begin{equation}\label{new3}
\mbox{\em $\psi(v^*)=r(v^*)$ and
$\psi(y)\neq y_{\varphi}$ for every $y\in N(S)$.
}
\end{equation}

Let $\theta$ be an $\HH$-map such that $\theta|_S = \varphi$ and $\theta|_{V(G)\setminus S} = \psi|_{V(G'-v^*)}$.
By~\eqref{new3},  $\theta$ is a $({\bf c}, \HH)$-coloring of $G$, a contradiction.
\end{proof}

Lemma~\ref{LM-M-1} implies that \begin{equation}\label{eq-a}
\text{for every  $F\subsetneq V(G)$ with $|F|\geq 2$, }  \rho_{G,{\bf c}}(F) \geq i-j+1.
\end{equation}

\begin{claimy}\label{-1-1}
    There is no edge $xy\in E(H)$, such that $\bc(x) = \bc(y) = -1$.
\end{claimy}
\begin{proof}
Suppose $uv\in E(G)$, $L(u) = \{u_1, u_2\}, L(v) = \{v_1, v_2\}$, and in $H$, $u_1\sim v_1, u_2\sim v_2$ in $H$, $\bc(u_1) = \bc(v_1) = -1$.
Then by $$i-j\leq\rho(G,{\bf c})\leq \rho_{G,\bc}(\{u,v\}) = \bc(u_2)+\bc(v_2)-1-1+2(i-j+1)-(i+1),$$
we get $\bc(u_2)+\bc(v_2)\geq j+1$, which implies $\bc(u_2), \bc(v_2)\geq 1$, since $\bc(x)\leq j$ for all $x\in V(H)$.
Let $\bc'$ differ from $\bc$ only in that $\bc'(u_2) = \bc(u_2)-1, \bc'(v_2) = \bc(v_2)-1 $.
Then $\rho_{\bc'}(G-uv) \geq i-j $.
Suppose not, let $S\subset V(G-uv)$ be a set with $\rho_{G-uv,\bc'}(S)\leq i-j-1$.
Then $S\cap \{u,v\}\neq \emptyset$.
If $u,v\in S$, then $\rho_{G,\bc}(S) = \rho_{G-uv,\bc'}(S) +2-(i+1)\leq i-j-1$, contradiction.
    If $u\in S, v\notin S$, then $\rho_{G,\bc}(S) = \rho_{G-uv,\bc'}(S) +1\leq i-j$. By Lemma~\ref{LM-M-1}, $S = \{u\}$, contradicts the fact that $\bc(u_2)\geq 1$.
    Hence, by symmetry, we may assume that $\phi$ is a  $\bc'$-coloring of $G-uv$. Then $\phi$ is also a $\bc$-coloring on $G$, a contradiction.
\end{proof}

The argument in Claim~\ref{-1-1} implies that 
\begin{equation}\label{fact2}
  \parbox{6in}{\em For each $uv\in E(G)$, with $\rho_\bc(u), \rho_\bc(v) \geq i-j+1 $, if  $\bc'$ differs from $\bc$ only in that for some $x\in L(u), y\in L(v)$, $\bc'(x) = \bc(x) - 1, \bc'(y) = \bc(y)-1$,  then $G-uv$ has a $\bc'$-coloring.}  
\end{equation}

\begin{claimy}\label{ngeq3}
    $|V(G)|\geq 3$.
\end{claimy}
\begin{proof} If $G$ has only one vertex, then for $G$ to be $\bc$-critical, the single vertex has to have capacity $(-1,-1)$, which implies that $\rho_\bc(G) = i-j-1$, a contradiction.

Now suppose $V(G)=\{u,v\}$ and  $G=K_2$. Say $L(u) = \{u_1, u_2\}, L(v) = \{v_1, v_2\}$, and $u_1\sim v_1, u_2\sim v_2$ in $H$.
If $\bc(u_k)$ and $\bc(v_{3-k})$ are both non-negative for either $k = 1$ or $2$, then we can color $G$  by letting $\varphi(u) = u_k, \varphi(v) = v_{3-k}$.
 Thus we may assume $\bc(u_2) = \bc(v_2) = -1$. But this contradicts Claim~\ref{-1-1}.
\end{proof}

\begin{lemma}\label{basics}
For every $u\in V(G)$,  ${\bf c_1}(u), {\bf c_2}(u)\geq 0$ and  $d(u)\geq 2$.
\end{lemma}
\begin{proof}
Suppose there is a vertex $u\in V(G)$ with  $L(u) = \{u_1, u_2\}$, where ${\bf c}(u_1) = -1$. Then $\bc(u_2)\geq 0$.
Choose such $u$ with the smallest $\bc(u_2)$.
Let $v\in N(u)$, $L(v) = \{v_1, v_2\}$ and $u_1v_1, u_2v_2\in E(H)$. 
By Claim~\ref{-1-1}, $\bc(v_1) \geq 0$. If $\bc(v_2) = -1$, then let $\phi$ be a $\bc$-coloring on $G-uv$. Since $\bc(u_1) = \bc(v_2) = -1$, $\phi(u) = u_2, \phi(v) = v_1$. Then $\phi$ is a $\bc$-coloring on $G$, a contradiction.

Now suppose $\bc(u_2)\geq 1$. Then by (\ref{fact2}), $G-uv$ has a $\bc'$-coloring $\phi$, where $\bc'(x) = \bc(x) - 1$ for $x\in \{u_2, v_2\}$. Then $\phi$ is a $\bc$-coloring on $G$, a contradiction.

Thus we may assume $\bc(u_2) = 0$.
If $d(v)=1$, then we find a $\bc$-coloring $\varphi$ of $G-v$ since $G$ is critical, and then let $\varphi(v)=v_1$. 

\vspace{2mm}

So we may assume
$d(v)\geq 2$.
 For a neighbor $w\neq u$  of $v$, denote $L(w) = \{w_1, w_2\}$ so that $w_1v_1, w_2v_2\in E(H)$. 
If ${\bf c}(w_1) = -1$ for all $w\in N(v)$, then we can easily extend a $\bc$-coloring $\varphi$ of $G-v$ to $G$ by letting $\varphi(v) = v_1$.
Thus we may assume some neighbor $w^*$ of $v$ has ${\bf c}(w^*_1)\geq 0$. Form $(G', {\bf \bc'})$, such that $G' = G-vw^*$, ${\bf \bc'}$ differs from ${\bf c}$ by only ${\bf \bc'}(x_1) = {\bf c}(x_1)-1$ for $x\in \{v,w^*\}$. 

If  $G'$ has a $\bc'$-coloring $\varphi$, then by definition, $\varphi(u) = u_2$,
and since $\bc(u_2) = 0$, $\phi(v) = v_1$.
Then $\phi$ is a $\bc$-coloring on $G$, a contradiction.
Otherwise, by (\ref{fact2}), $\rho_\bc(w) = i-j$. 
Then $\rho_{G,\bc}(\{u,v,w\}) = \rho_\bc(u) + \rho_\bc(v) + \rho_\bc(w) - 2(i+1) \leq 2(i-j)+2i+1-2(i+1) = i-j-1 + (i-j) < i-j-1$, a contradiction to the choice of $G$.
\vspace{1mm}
\\
This proves the first half of the statement. Now if some $u$ has degree $1$, since ${\bf c_1}(u), {\bf c_2}(u)\geq 0$, we can extend any coloring of $G-u$ to $u$ greedily, a contradiction.
\end{proof}

Lemmas~\ref{LM-M-1} and \ref{basics}  imply that
\begin{equation}\label{eq-b}
\text{for every  $\emptyset \neq F\subsetneq V(G)$, } \rho_{G,{\bf c}}(F) \geq i-j+1.
\end{equation}

\begin{corollary}\label{cor1}
    Let $\emptyset\neq S\subset V(G)$. Then for all $A,B\subset V(G)\setminus S$ with $\rho(A) = \rho(B) = i-j+1$, 
$$
\rho(A\cup B) = \rho(A\cap B) = i-j+1.
$$
\end{corollary}

\begin{proof}
By submodularity,
$
\rho(A) + \rho(B) \geq \rho(A\cup B) + \rho(A\cap B).
$
Also none of $A,B, A\cup B, A\cap B$ is equal to $V(G)$. So the claim follows from (\ref{eq-b}).
\end{proof}

A helpful notion in this and next sections is the notion of $B(S)$:
For $S\subset V(G)$, denote by $B(S)$ the union of all the subsets in $V(G)\setminus S$ with potential equal to $i-j+1$. When $S$ consists of only a single vertex $v$, we write $B(v)$ instead of $B(\{v\})$ for simplicity. By Corollary~\ref{cor1}, we have the following.

\begin{corollary}\label{cor61}
For every nonempty $S\subsetneq V(G)$, if $B(S)\neq\emptyset$, then $\rho_{G,\bc}(B(S)) = i-j+1$.
\end{corollary}

\begin{claim}\label{CL2}
{\em Let $xy\in E(G)$, $L(x) = \{x_1, x_2\}, L(y) = \{y_1, y_2\}$, $x_1\sim y_1, x_2\sim y_2$.
For $h=1,2$, graph $G-xy$ has a $\bc$-coloring $\psi_h$ such that $\psi_h(x) = x_h$.}

\smallskip
Indeed, let $G'=G-xy$ and let $\bc'$ differ from $\bc$ only in that $\bc'(x_{3-h})=\bc(x_{3-h})-1$
and $\bc'(y_{3-h})=\bc(h_{3-h})-1$. By (\ref{eq-b}), if $A\subseteq V(G)$ and $|A\cap \{x,y\}|\leq 1$, then
$\rho_{G',\bc'}(A) \geq \rho_{G,\bc}(A)-1 \geq 2-j$. On the other hand if $ \{x,y\}\subseteq A$, then
$\rho_{G',\bc'}(A) \geq \rho_{G,\bc}(A)-2+(i+1) \geq \rho_{G,\bc}(A)\geq 2-j$.
 So, by the minimality of $G$, $G'$ has a $\bc'$-coloring $\psi_h$.

If $\psi_h(x)=x_{3-h}$, then by the definition of $\bc'$,   $\psi_h$ is a $\bc$-coloring of $G$ regardless of the value of
$\psi_h(y)$, a contradiction. Thus, $\psi_h(x) = x_h$, as claimed.
\end{claim}

We say a vertex $v\in V(G)$ is a \emph{$(d;c_1,c_2)$-vertex} if $d(v)=d$, ${\bf c}_1(v)=c_1$, and ${\bf c}_2(v)=c_2$.
We call a $(2;i,j)$-vertex a {\em surplus vertex}. All non-surplus vertices will be called {\em ordinary}.
Let $V_0=V_0(G,{\bf c})$ denote the set of surplus vertices in $(G,{\bf c})$.

\begin{observation}\label{ob1}
Surplus vertices  in $(G,\bc)$ cannot be adjacent.
\end{observation}
\begin{proof}
Suppose $v_1,v_2\in V_0(G)$ and $v_1v_2\in E(G)$.
Let $v'_j$ be the  neighbor of $v_j$ distinct from $v_{3-j}$.
Since $G$ is ${\bf c}$-critical, $G-v_1-v_2$ has a ${\bf c}$-coloring $\varphi$. Extend $\varphi$ to $v_1,v_2$ by choosing $\varphi(v_j)\neq \varphi(v'_j)$ for $j=1,2$. Then $\varphi$ is a ${\bf c}$-coloring on $G$, a contradiction.
\end{proof}

To show that $(G,{\bf c})$ does not exist, we use discharging. It is done in two steps. At the beginning, the charge of each vertex
$v$ is $\rho_\bc(v)$ and the charge of each edge is $-(i+1)$, so by~\eqref{eq-1}, ~\eqref{eq-a} and Lemma~\ref{basics}, the total sum of all charges is
$\rho(G,{\bf c})$.
On Step 1, each edge gives to each its end charge $-(i+1)/2$ and is left with charge $0$. Note that each surplus vertex after Step~1 has charge $2i+1-2\frac{i+1}{2}=i$. On Step 2, each $u\in V_0$ gives $i/2$ to each of its two neighbors and is left with charge $0$. The resulting charge, $ch$, can be nonzero only on vertices in $V(G)-V_0$, and the total charge over all vertices and edges does not change. Thus,
\begin{equation}\label{total}
\sum_{v\in V(G)-V_0} ch(v)=   \sum_{v\in V(G)} ch(v)=\rho(G,{\bf c}).
    \end{equation}
For an ordinary vertex $v$, if $d_1(v)$ denotes the number of ordinary neighbors of $v$ and $d_2(v)$ denotes the number of surplus neighbors of $v$, then
\begin{equation}\label{charge}
    ch(v) = \rho_\bc(v)-\frac{i+1}{2}d_1(v)-\frac{1}{2}d_2(v) = \bc_1(v)+\bc_2(v)+i-j+1-\frac{i+1}{2}d_1(v)-\frac{1}{2}d_2(v).
    \end{equation}

The following lemma is the crucial final step of our proofs.

\begin{lemma}\label{crutial}
    For $i\geq 1, j\geq 2i$, if $ch(v)\leq 0$ for all $v\in V(G)$, then $G$ is $\bc$-colorable.
\end{lemma}
\begin{proof} Suppose $ch(v)\leq 0$ for all $v\in V(G)$.
Construct $G_h$ and $H_h$ from $G$ and $H$ as follows:

Forming $G_h$: for each surplus vertex $v$, say $N(v) = \{u,w\}$, we replace $v$ with an edge connecting $u$ and $w$. We call such edge a \emph{half edge}. And if
 $L(x) = \{x_1, x_2\}$ for $x\in \{u,v,w\}$ so that $u_1v_1, w_1v_1\in E(H)$, then for the
 cover graph $H_h$   we delete $L(v)$ from $H$ and add edges $u_1w_2, u_2w_1$.
Note that $G_h$ and $H_h$ are not necessarily simple graphs. 

    Since $\rho(G,{\bf c})= \sum_{v\in V(G_h)}ch(v)\geq i-j$  by assumption and all charges are nonpositive, 
    \begin{equation}\label{nonpo}
 \mbox{\em $ch(v)\geq i-j$ for every vertex $v$, and equality may hold for at most one vertex. 
 }           \end{equation}
    
Let $\varphi$ be a $\HH_{G_h}$-map (does not need to be a coloring).
Define\\ $d^*_\varphi (v) = |\{ uv: uv \text{ not a half edge}, \varphi(u)\sim \varphi(v) \}| + \frac{1}{2} |\{ uv: uv \text{ is a half edge}, \varphi(u)\sim \varphi(v) \}| $,
$ S(\varphi):= \sum_{v\in V(G)} \bc(\varphi(v)) - \frac{1}{2}\sum_{v\in V(G)} d^*_\varphi (v)$. 

Let $S:= \max\{S(\varphi) :\varphi\text{ is a }\HH_{G_h}\text{-map} \}$, and $\psi$ a $\HH_{G_h}$-map with $S(\psi) = S$. 

Suppose $\bc(\psi(u))< d^*_\psi (u)$ for some $u\in V(G_h)$. Let $\psi_u$ differ from $\psi$  only on $u$. 
By the choice of $\psi$, 
$S(\psi_u) \leq S(\psi) $. Hence
$$0\geq S(\psi_u) - S(\psi) =  \bc(\psi_u(u))- d^*_{\psi_u} (u) - (\bc(\psi(u))- d^*_\psi (u)).$$
So, $\bc(\psi_u(u))- d^*_{\psi_u} (u) \leq \bc(\psi(u))- d^*_\psi (u)$, and
\begin{multline*}
 2(\bc(\psi(u))- d^*_\psi (u))\geq \bc(\psi(u))- d^*_\psi (u) + \bc(\psi_u(u))- d^*_{\psi_u} (u)
=\bc_1(u)+\bc_2(u)-d_1(u)-\frac{d_2(u)}{2} \\
= ch(u)-i+j-1+\frac{i-1}{2}d_1(u)
\geq i-j -i+j-1+\frac{i-1}{2}d_1(u)
\geq -1 .
\end{multline*}

Also since $2(\bc(\psi(u))- d^*_\psi (u))  $ is an integer, we must have $ch(u) = i-j$ and $\bc(\psi(u))- d^*_\psi (u) =\bc(\psi_u(u))- d^*_{\psi_u} (u) = -1/2 $.
And by~\eqref{nonpo}, there is at most one such $u$.

We say an edge $xy\in E(G_h)$ is \emph{$\psi$-conflicting}  if $\psi(x)\sim \psi(y)$. Let $G'$ be the spanning subgraph of $G_h$ where $E(G')$ consists of only the $\psi$-conflicting half edges in $G$ under $\psi$. 

Since $\bc(\psi(u))-d^*_\psi(u) = -1/2$, $ d_{G'}(u)>0$ and is an odd number. Let $C$ be the component in $G'$ containing $u$. Then there is another vertex $v\in C$ with $d_{G'}(v)$ odd.
And since $u$ is the unique vertex with $\bc(\psi(u))< d^*_\psi (u)$, $\bc(\psi(v))-d^*_\psi(v)\geq 1/2$.
Let $P$ be a $uv$-path in $G'$. 
Let $G'' = G'-E(P)$.
Add a vertex $v^*$ to $G''$ and add an edge between $v^*$ and every odd-degree vertex in $G''$. Then we can decompose $E(G'')$ into cycles. Let $\tau$ be an Eulerian orientation on these cycles. 
Extend $\tau$ to $E(P)$ so that $P$ is a directed path from $v$ to $u$.

We extend $\psi$ from $G_h$ to $G$ as follows.
For a surplus vertex $x$, if it does not correspond to a $\psi$-conflicting edge in $G_h$, color $x$ so that $\varphi(x)$ is adjacent to neither of its neighbors; if $x$ corresponds to a $\psi$-conflicting vertex, then we color $x$ so that $\psi(x)$ is not adjacent to the head of the corresponding half edge w.r.t. $\tau$. Then $\psi$ is a $\bc$-coloring of $G$, a contradiction.

If $\bc(\psi(x))-d^*_\psi(x) \geq 0$ for every $x\in V(G_h)$, then we extend $\psi$ to $G$  as above, just take $G'' = G'$ in the above construction.
\end{proof}

In the following two sections, we will prove that for each $i = 1,2, j\geq 2i$, $ch(v)\leq 0$ for every $v\in V(G)$ (assume for each pair of $(i,j)$ with $i = 1,2, j\geq 2i$, $(G,\bc)$ denotes the minimal counterexample).
Then together with Lemma~\ref{crutial}, we will get a contradiction to the choice of $G$.

\section{The case of $i = 1$ and $j\geq 2$}\label{1,geq2}

In this section, the potential of a vertex  with capacity $(c_1, c_2)$ is $ c_1+c_2+2-j$, and 
the potential of an edge is $- 2$. Our  $G$ has potential at least $1-j$. And by (\ref{eq-b}), the potential
of each proper nonempty subset of $V(G)$  is at least $2-j$. To show that $G$ has no vertices with positive charge, we 
first prove four claims on potentials of vertices with small degree.




\begin{claimy}\label{d2not12}
 If $u\in V(G)$ is a degree $2$ vertex, then $\rho_\bc(u)\neq 2$.
\end{claimy}

\begin{proof}
Suppose  $N_G(u) = \{v,w\}$ and $\rho(u) = 2$.
By symmetry, we may assume ${\bf c}(u_2)\geq {\bf c}(u_1)$.
For $x\in \{u,v,w\}$ let $L(x) = \{x_1, x_2\}$ be such that $y_1u_1, y_2u_2\in E(H)$ for $y\in N(u)$.
\vspace{1mm}
\\
\textbf{Case 1.} ${\bf c}(u_1) = 0, {\bf c}(u_2) = j$. Form $(G', {\bf c'})$, such that $G' = G-u$ and ${\bf c'}$ differs from ${\bf c}$ in $G'$ by only ${\bf c'}(x_2) = {\bf c}(x_2)-1$ for $x\in N(u)$. 
If there is  $A\subset V(G')$ with $\rho_{G',{\bf c'}}(A)\leq -j$, then by (\ref{eq-b}), $v,w\in A$. But then $\rho_{G,{\bf c}}(A\cup\{u\})= \rho_{G',{\bf c'}}(A) +2+2-2\cdot 2 \leq -j $, a contradiction.
\vspace{1mm}

\noindent\textbf{Case 2.} ${\bf c}(u_1) = 1,  {\bf c}(u_2) = j-1$. For each $x\in N(u)$, form $(G', {\bf c_x})$ so that $G' = G-u$, ${\bf c_x}$ differs from ${\bf c}$ by only ${\bf c_x}(x_i) = {\bf c}(x_i)-1$ for $i = 1,2, x\in N(u)$. If there is some $S_x\subset V(G')$ such that $\rho_{G',{\bf c_x}}(S_x)\leq -j$, then $x\in S_x$. If $N(u)\subset S_x$, then $\rho_{G,{\bf c}}(S_x\cup\{u\})= \rho_{G',{\bf c_x}}(S_x)+2+2-2\cdot 2\leq -j$, a contradiction. 
Thus $S_v\cap N(u) = \{v\}, S_w\cap N(u) = \{w\}$, and $\rho_{G, {\bf c}}(S_x)\leq 2-j$ for $x\in N(u)$.
By submodularity and (\ref{eq-b}), $\rho_{G,{\bf c}}(S_v\cup S_w)\leq \rho_{G, {\bf c}}(S_v)+\rho_{G, {\bf c}}(S_w)\leq 2-j$. Then $\rho_{G,{\bf c}}(S_v\cup S_w\cup\{u\})\leq 2-j+2-2\cdot 2 = -j$, a contradiction.
\end{proof}


\begin{claimy}\label{d3}
Let $u\in V(G)$ be a degree three vertex. If $u$ has at least one surplus neighbor, then $\rho_\bc(u)\leq 1$.
\end{claimy}
\begin{proof}
Suppose $N(u)=\{x,y,v\}$, $N(v)=\{u,v'\}$, $v$ is a surplus vertex and $\rho_\bc(u)\geq 2$.
 For $w\in N(v)\cup N(u)$, let $L(w) = \{w_1, w_2\}$ so that $w_1z_1, w_2z_2\in E(H)$ whenever $w\sim z$ in $G$. By symmetry, assume ${\bf c}(u_1)\leq {\bf c}(u_2)$.
\vspace{1mm}
\\
\textbf{Case 1.} ${\bf c}(u_2) \geq 2$. 
Form $(G',{\bf c'})$ as follows: $G' = G-u-v$ and ${\bf c'}$ differs from ${\bf c}$ only by ${\bf c'}(z_2) = {\bf c}(z_2)-1$ for $z\in \{x,y\}$.
If some $S\subset V(G')$ has $\rho_{G',{\bf c'}}(S) \leq -j$, then $x,y\in S$. But then $\rho_{G,{\bf c}}(S\cup\{u\})= \rho_{G',{\bf c'}}(S)+2 +\rho_{\bf c}(u)-2\cdot 2\leq  1-j$ if $v'\notin S$, and $\rho_{G,{\bf c}}(S\cup\{u,v\}) \leq \rho_{G',{\bf c'}}(S)+2+3\cdot 2-2\cdot 4\leq -j$ if $v'\in S$. Neither case is possible by (\ref{eq-b}).
Thus  $G'$ has a ${\bf c'}$-coloring $\varphi$.
Extend $\varphi$ to $G$: first let  $\varphi(v)\nsim \varphi(v')$, and then let $\varphi(u)=u_2$, unless all three neighbors of $u_2$ are colored by $\varphi$, in which case  let $\varphi(u)=u_1$.
\\
\textbf{Case 2.} ${\bf c}(u_2) = 1$. Then $j=2$, ${\bf c}(u_1) = 1$ and $\rho_{\bf c}(u) = 2$. Still, let $G' = G-u-v$. Define ${\bf c''}$ so that ${\bf c''}$ differs from ${\bf c'}$ by only ${\bf c''}(v'_1) = {\bf c'}(v'_1)-1$. If some $S\subset V(G')$ has $\rho_{G',{\bf c''}}(S) \leq -j$, then $|S\cap \{x,y,v'\}|\geq 2$.
If $\{x,y,v'\}\subset S$, then
$\rho_{G,{\bf c}}(S\cup\{u,v\}) = \rho_{G',{\bf c''}}(S)+3+\rho_{\bf c}(u)+ \rho_{\bf c}(v)-4\cdot 2\leq -j$, a contradiction.
If $S\cap \{x,y,v'\}= \{x,v'\}$, then 
$\rho_{G,{\bf c}}(S\cup\{u,v\}) = \rho_{G',{\bf c''}}(S)+2+\rho_{\bf c}(u)+ \rho_{\bf c}(v)-3\cdot 2\leq 1-j$, contradicting  (\ref{eq-b}). The remaining possibilities $S\cap \{x,y,v'\}= \{y,v'\}$ and $S\cap \{x,y,v'\}= \{x,y\}$ are very similar.

Thus  $G'$ has a ${\bf c''}$-coloring $\psi$.
 Extend $\psi$ to $G$: 
if $\psi(x) = x_1, \psi(y) = y_1$, let $\psi(u) = u_2$ and $\psi(v)\nsim \psi(v')$, then $\psi$ is a ${\bf c}$-coloring on $G$.  The case when $\psi(x) = x_2, \psi(y) = y_2$ is similar.
Then $\psi(x) = x_1, \psi(y) = y_2$ or  $\psi(x) = x_2, \psi(y) = y_1$. Let $\psi(u) = u_2$ and $\psi(v)=v_1$. Then $\psi$ is a ${\bf c}$-coloring of $G$, a contradiction.
\end{proof}



\begin{claimy}\label{d4}
A $(4;1,j)$-vertex cannot have three surplus neighbors.
\end{claimy}
\begin{proof}
Suppose the claim fails for some $(4;1,j)$-vertex $u$. Let $N(u)=\{x,y,z,v\}$ where $x,y,z$ are surplus vertices and $x',y',z'$ are their other neighbors. Note that $v$ may also be a surplus vertex.
Suppose ${\bf c}(u_1) = 1, {\bf c}(u_2) = j$
For $w\in N(u)$, denote $L(w) = \{w_1,w_2\}$  so that   $u_1w_1, u_2w_2\in E(H)$, and
for $w\in N(u)-v$, denote $L(w') = \{w'_1,w'_2\}$  so that   $w_1w'_1, w_2w'_2\in E(H)$.

\vspace{1mm}
\noindent\textbf{Case 1.} $v\notin B(u)$. Form $(G', {\bf c'})$ so that $G' = G- \{u,x,y,z\}$ and ${\bf c'}$ differs from ${\bf c}$ by only ${\bf c'}(v_k) = {\bf c}(v_k)-1$ for $k = 1,2$. Since $v\notin B(u)$,
 $\rho_{G',\bc'}(S)\geq 2-j$ for each $S\subseteq V(G')$ by
 \eqref{eq-b}; thus by the minimality of $G$, $G'$ has a $\bc'$-coloring $\varphi$.
 We extend $\varphi$ to $G$ as follows. First, for $w\in \{x,y,z\}$ we let $\varphi(w)\nsim \varphi(w')$. Second, if the color of
at most one vertex in $\{v,x,y,z\}$ conflicts with $u_1$, then we let $\varphi(u)=u_1$, else the color of
at most two vertices in $\{v,x,y,z\}$ conflicts with $u_2$, and we let $\varphi(u)=u_2$.

\vspace{1mm}
\noindent\textbf{Case 2.} $\{v,x',y',z'\}\subseteq B(u)$. By Corollary~\ref{cor61},  the following contradicts the choice of $G$:
$$\rho_{G,\bc}(B(u)\cup \{u,x,y,z\})\leq \rho_{G,\bc}(B(u))+4\cdot 3-7\cdot 2\leq (2-j)-2=-j.$$ 
\textbf{Case 3.}
$v\in B(u)$ and there is $w\in \{x,u,z\}$ such that $w'\notin B(u)$, say $x'\notin B(u)$. Then $x'\neq v$.
Form $(G',{\bf c''})$: let ${\bf c''}$ differ from ${\bf c}$ only by ${\bf c''}(v_2) = {\bf c}(v_2)-1$ and
${\bf c''}(x'_1) = {\bf c}(x'_1)-1$. Since $x'\notin B(u)$,
 $\rho_{G',\bc''}(S)\geq 2-j$ for each $S\subseteq V(G')$ by
 \eqref{eq-b}, and so  $G'$ has a $\bc''$-coloring $\psi$. We extend $\psi$ to $G$ as follows. First, 
  for $w\in \{y,z\}$ we let $\psi(w)\nsim \psi(w')$. If at most two 
$w\in \{v,y,z\}$ are colored with $w_2$, then we let   $\psi(x)=x_1$  and $\psi(u)=u_2$, else we let
$\psi(x)\nsim \psi(x')$ and  $\psi(u)=u_1$.
\end{proof}

\begin{claimy}\label{d5}
A $(5;1,j)$-vertex cannot have five surplus neighbors.
\end{claimy}
\begin{proof}
Suppose $u$ is a $(5;1,j)$-vertex with all its five neighbors being surplus vertices.
For every $v\in N(u)$, let $v'$ be the neighbor of $v$ distinct from $u$. Let $T=N(u)\cup \{u\}$ and $T'=\{v': v\in N(u)\}$.
Consider the graph $G' = G-T$. 
If $T'\subseteq B(T)$, then $\rho_{G,\bc}(B(T)\cup T)\leq \rho_{G,\bc}(B(T))+6\cdot 3-10\cdot 2\leq 2-j-2$, a contradiction.
Thus, there is $x\in N(u)$ such that $x'\notin  B(T)$. Denote $L(x') = \{x'_1,x'_2\}$.

Define ${\bf c'}$ so that ${\bf c'}$ differs from ${\bf c}$ by only ${\bf c'}(x'_k) = {\bf c}(x'_k)-1$ for $k=1,2$. 
Since $x'\notin B(T)$,
 $\rho_{G',\bc'}(S)\geq 2-j$ for each $S\subseteq V(G')$ by
 \eqref{eq-b}; thus  $G'$ has a $\bc'$-coloring $\varphi$.
For every $w\in N(u)\setminus\{x\}$, extend $\varphi$ to $w$ so that $\varphi(w)\nsim \varphi(w')$.
If $r(u)$ has at most two colored neighbors in $N(u)-x$, then we let $\varphi(u)=r(u)$, else 
$p(u)$ has at most $4-3=1$ colored neighbor in $N(u)-x$, and we let $\varphi(u)=p(u)$. In both cases,
we let $\varphi(x)\nsim \varphi(u)$. By construction, we get a ${\bf c}$-coloring of $G$.
\end{proof}

Now we are ready to prove the main lemma of this section.

\begin{lemma}\label{charge0}
For every vertex $v\in V(G)$, $ch(v)\leq 0$.
\end{lemma}

\begin{proof} Let $v\in V(G)$ and $d_G(v)=d$. By Lemma~\ref{basics}, $d\geq 2$. 

If $d=2$, then by Claim~\ref{d2not12},
either  $v$ is a surplus vertex and has $ch(v) =  0$ by definition, or $\rho(v)\leq 1$  and 
  $ch(v)\leq 1-2\cdot(1/2) = 0$.
  
If $d=3$, then by Claim~\ref{d3},  either $v$ has no surplus neighbors and so $ch(v)\leq 3-3 = 0$, or
$\rho(v)\leq 1$ and  $ch(v)\leq 1-3\cdot (1/2) = -1/2$.

If $d=4$, then by Claim~\ref{d4},  either $v$ has at most two surplus neighbors and so $ch(v)\leq 3-2- 2\cdot(1/2) = 0$, or
$\rho(v)\leq 2$ and  $ch(v)\leq 2-4\cdot (1/2) = 0$. 

If $d=5$, then by Claim~\ref{d5},  either $v$ has at most four surplus neighbors and so $ch(v)\leq 3-1- 4\cdot(1/2) = 0$, or
$\rho(v)\leq 2$ and  $ch(v)\leq 2-5\cdot (1/2) = -1/2$. 

If $d\geq 6$, then $ch(v)\leq 3-6\cdot (1/2) = 0$.
\end{proof}

Now Lemma~\ref{crutial} together with Lemma~\ref{charge0} imply the part $i=1$ of
Theorem~\ref{MT-F}.

\section{The case of $i = 2$ and $j\geq 4$.}
In this section, the potential of a vertex  with capacity $(c_1, c_2)$ is $ c_1+c_2+3-j$, and 
the potential of an edge is $-(i+1) = -3$. Our  $G$ has potential at least $2-j$. And by (\ref{eq-b}), the potential
of each proper nonempty subset of $V(G)$  is at least $ 3-j$.

\begin{lemma}\label{nsnbg}
 Suppose there is a partition  of $V(G)$ into nonempty sets $F,S$ and $R$ such that each vertex in $S$ is a surplus vertex with  one neighbor in  $F$ and one neighbor in $R$, and there is no edge connecting $F$ with $R$. Then $\rho_{G,\bc}(F)> 0$.
\end{lemma}
\begin{proof} Suppose $\rho_{G,\bc}(F)\leq 0$.
For every vertex $x\in S$, denote by $x_F$ (respectively, $x_R$)
 the  neighbor of $x$ in $F$ (respectively, in $R$). We say that $\alpha\in L(x_R)$ {\em is conflicting} with $\beta\in L(x_F)$
 if their neighbors in $L(x)$ are distinct.

Since $G$ is $\bc$-critical, $G[F]$ has
 a $\bc$-coloring  $\varphi$. We obtain a new capacity function $\bc'$ on $R$ from  $\bc$ as follows. For every  $x\in S$,
decrease the capacity of the node in $L(x_R)$ conflicting with $\varphi(x_F)$ by $1$. If a vertex  $y\in R$ is adjacent to $s$ vertices in $S$, then such decrease  for nodes in $L(y)$ will happen $s$ times. If $G[R]$ has a $\bc'$-coloring  $\varphi'$, then we extend
 $\varphi$ to each $x\in S$ by $\varphi(x)\nsim \varphi(x_F)$, and now $\varphi\cup \varphi'$ will be a 
 $\bc$-coloring of $G$ by the choice of $\bc'$. Thus $G[R]$ has no $\bc'$-coloring.
 
By the minimality of $G$,  there is some $R'\subset R$ with $\rho_{G,\bc'}(R')\leq 1-j$. 
Let $S'\subset S$ be the set of surplus vertices connecting $R'$ with $F$ in $G$.
Since $\rho_{G,\bc}(F)\leq 0$,  
$$\rho_{G,\bc}(F\cup R'\cup S')= \rho_{G,\bc}(F)+\rho_{G,\bc}( R')- |S'| \leq 0+(1-j+|S'|)  - |S'| \leq 1-j,$$
a contradiction.
\end{proof}

A set $A\subset V(G)$ in $(G,{\bf c})$ is {\em trivial} if $A=V(G)$ or $V(G)-A$ is a surplus vertex,
and {\em nontrivial} otherwise.

\begin{lemma}\label{geq0for24}
For any  nontrivial $F\subset V(G)$, $\rho_{G,\bc}(F)\geq 4-j$, or $F$ is a single vertex with potential $3-j$.
\end{lemma}
\begin{proof}
Suppose there is a nontrivial $F\subset V(G)$ with $\rho_{G,\bc}(F) \leq 3-j$ and $|F|>1$. 
Choose a maximum such $F$.  
If some vertex $w\in V(G)-F$ has at least two neighbors in $F$, then consider $F'=F+w$. Since $\rho_{G,\bc}(F')\leq \rho_{G,\bc}(F)+\rho_{G,\bc}(w)-3d(w)<\rho_{G,\bc}(F)$, the maximality of $F$ implies that $F'$ is trivial, which means $F'=V(G)$ or $F'=V(G)-z$ for some surplus vertex $z$. If $F'=V(G)$, then since $F$ is nontrivial, $\rho_{G,\bc}(V(G))\leq \rho_{G,\bc}(F)-2\leq 1-j$, a contradiction.
If $F'=V(G)-z$ for some surplus vertex $z$, then $\rho_{G,\bc}(F')\leq \rho_{G,\bc}(F)-1\leq 2-j,$ contradicting (\ref{eq-b}).
Thus every vertex in $V(G)\setminus F$ has at most one neighbor in $F$. So, if all vertices in the set $S=N(F)-F$ were surplus vertices,
then the set $R=V(G)\setminus(F\cup N(F))$ is nonempty, because $S$ is independent. Thus, the sets $F,S$ and $R$ would contradict
Lemma~\ref{nsnbg}. Therefore, $N_G(F)-F$ has an ordinary vertex, say $y$.

Let $x$ be the neighbor of $y$ in $F$. We can change the names of the colors so that
\begin{equation}\label{x1y1}
  \mbox{\em $L(x) = \{x_1, x_2\}, L(y) = \{y_1, y_2\}$, 
  $\bc(y_2)\geq \bc(y_1)$, $x_1\sim y_1$ and $x_2\sim y_2$.
  }  
\end{equation}

\begin{claim}\label{Claim2}
{\em Every neighbor of $y$ outside of $F$ is a surplus vertex.}

\smallskip
Indeed, suppose $z$ is an ordinary neighbor of $y$ outside of $F$, say $ L(z) = \{z_1, z_2\}$, where $y_1\sim z_1$. 
Since $G$ is $\bc$-critical, $G[F]$ has 
 a $\bc$-coloring $\varphi$;  say $\varphi(x) = x_1$. Construct $G', \bc', H'$ from $G,\bc,H$  as follows:

Replace $F$ by a single vertex $v$, where $L(v) = \{v_1, v_2\}$ and $\bc'(v_1) = 0, \bc'(v_2) = -1$. For every vertex $u\in N_G(F)-F$, denote $L(u) = \{u_1, u_2\}$ so that the neighbor of $u_1$ in $H$ is colored by $\varphi$. In $G'$, add an edge between $v$ and each vertex in $N_G(F)$. In $H'$, let $v_1u_1, v_2u_2\in E(H')$. 
Remove edge $yz$ from $E(G')$, also remove edges between $L(y)$ and $L(z)$ in $H'$. 
Let $\bc'(y_2) = \bc(y_2)-1, \bc'(z_2) = \bc(z_2)-1$, and let $G', \bc', H'$ agree with $G, c, H$ everywhere else.

If $G'$ has a 
$\bc'$-coloring $\psi$, then  $\psi(v) = v_1$ and $ \psi(u) = u_2$ for all $u\in N_G(F)-F$. 
Let $\theta$ be an $H$-map such that $\theta = \varphi$ on $F$ and $\theta = \psi$ on $V(G)\setminus F$. Then $\theta$ is a $\bc$-coloring on $G$, a contradiction. Thus $G'$ has no
$\bc'$-coloring. 
Since $G'$ is a smaller graph than $G$, there is
 $S\subset V(G')$ with $\rho_{G',\bc'}(S)\leq 1-j$. Since $\rho_{G',\bc'}(v)=2-j\leq -1,$ we may assume
  $v\in S$. Let $S' = (S-v)\cup F \subset V(G)$. 
If $y,z\notin S$, then $\rho_{G,\bc}(S') \leq \rho_{G',\bc'}(S) +1  \leq 2-j$, a contradiction to (\ref{eq-b}).
If exactly one of $y,z$ is in $S$, then $\rho_{G,\bc}(S') \leq \rho_{G',\bc'}(S) +1 +1 \leq 3-j$, and $S'\supset F$. Since one of $y,z$ is not in $S'$, this means that $S'$ is  a larger than $F$ nontrivial set with potential at most $3-j$,
 a contradiction. 
Thus $\{v,y,z\}\subseteq S$, and  $\rho_{G,\bc}(S') \leq \rho_{G',\bc'}(S) +1+1+1-3\leq 1-j$, a contradiction again.
\end{claim}

Let $u_1,\dots, u_d\in N(y)$ be the  surplus neighbors of $y$ outside of $F$, and $u'_1,\dots, u'_d$ be their 
 other neighbors, where $u'_i$s are not necessarily distinct. 

\begin{claim}\label{CL2a}
{\em $d>\bc(y_2)$.}

\smallskip Indeed, suppose $d\leq\bc(y_2)$.
 By Claim~\ref{CL2}, there is
 a $\bc$-coloring $\varphi$ of $G-y$ with $\varphi(x) = x_1$. We recolor
each $u_h$ so that it has no conflict with $u'_h$, and then color $y$
with $y_2$. This would give a $\bc$-coloring on $G$, a contradiction.
\end{claim}

\begin{claim}\label{CL3}
{\em $\rho_\bc(y) = 5$, and hence $\bc(y) = (2,j)$.}

\smallskip
If $\rho_\bc(y)\leq 3$, then $\rho_{G,\bc}(F+y)\leq 3-j-3+3=3-j$, a contradiction to the maximality of $|F|$. Suppose $\rho_\bc(y) = 4$.
 By Claim~\ref{CL2}, there is
 a $\bc$-coloring $\varphi$ of $G[F]$ with $\varphi(x) = x_1$.

Since $\rho_\bc(y) = 4$, $\bc(y_2)\geq j-1\geq 3$.
Construct $G'', c''$ as follows:
Let
$G'' = G-F+v-y-u_1-\cdots -u_d$, where $v$ is the same as in the proof of Claim~\ref{Claim2}. Let $\bc''$ differ from $\bc$ only in that
for each $1\leq h\leq d-\bc(y_2)$,
 the capacity of  $u'_h$ decreases  by $1$. This means that if some $u'_{h_1},\ldots,u'_{h_s}$ coincide, then the capacity of the corresponding vertex decreases by $s$.

Suppose some $S\subset V(G'')$ has $\rho_{G'',\bc''}(S)\leq 1-j$. 
Let $S$ be a maximal one with this property.
Then $v\in S$.
Let $a$ be the number of indices  $1\leq h\leq d-\bc(y_2)$ such that $u'_h$s in $S$ and 
$b$ be the number of indices  $d-\bc(y_2)+1\leq h\leq d$ such that $u'_h$s in $S$.
Denote by
 $S'$ the set obtained from $S-v+F+y$ by adding the surplus vertices in $N(y)$ connecting $S-v$ with $y$ in $G$.
Then $\rho_{G,\bc}(S') \leq 1-j+a -a -b +\rho_\bc(y) -3 +1 = \rho_\bc(y)-b-1-j$.
If $\rho_\bc(y) = 4$, then $\rho_{G,\bc}(S') \leq 3-j$. By the maximality of $|F|$, $S'$ is trivial. But then $b = \bc(y_2)\geq 3$, $\rho_{G,\bc}(S')\leq -j$, a contradiction.

Thus by the minimality of $G$, graph $G''$ has a $\bc''$-coloring $\psi$. By the definition of $v$, $\psi(v) = v_1$. We extend $\psi$ to $y$ and $u_i$'s, so that $\psi(y) = y_2$, and $y_2$ has at most $\bc(y_2)$ neighbors in $\psi(u_1),\dots, \psi(u_d)$. 
Then $\psi\cup \varphi$ is a $\bc$-coloring on $G$, a contradiction.
\end{claim}

Now we prove the lemma.

By Claim~\ref{CL3}, $\bc(y_2) = j$.
By Claim~\ref{CL2}, there is
 a $\bc$-coloring $\varphi$ of $G[F]$ with $\varphi(x) = x_1$. 

We  construct $G'', \bc''$ as in Claim~\ref{CL3}.
Let $N'_1 \subset N'$ be the (multi)set of secondary neighbors whose  capacity decreased, and let $N'_2 = N'\setminus N'_1$. 
By Claim~\ref{CL2a},
 $d>j$.
So, (as a multiset) $|N'_1| = d-j$, $|N'_2| = j$.

Note that $|N'\cap F|\leq 1$, since otherwise $\rho_{G,\bc}(F+y)\leq 3-j+5-3-2 = 3-j$, which contradicts the choice of $F$.

  As in the proof of Claim~\ref{CL3}, there is some $S\subset V(G'')$ with $\rho_{G'',\bc''}(S)\leq 1-j$. Define $S' ,a,b$ as in Claim~\ref{CL3}. Then $\rho_{G,\bc}(S')\leq 1-j+a-a-b+\rho_\bc(y)-3+1 =4-j -b$.
If $b>0$, then $\rho_{G,\bc}(S')\leq 4-j -b\leq 3-j$.
So
by the maximality of $|F|$, $S'$ is trivial, and hence
$b = \bc(y_2) = j\geq 4$, implying $\rho_{G,\bc}(S')\leq 4-j -b\leq -j$ a contradiction. Thus $b = 0$, that is, $S'\cap N'_2 = \emptyset$, and $\rho_{G,\bc}(S')\leq 4-j$. 

By choosing different $N'_2\subset N'$, we can form different corresponding $S'$. 
Let $\mathcal{S}$ denote the family of all  $S'\subset V(G)$ satisfying:
(i) $\rho_{G,\bc}(S')\leq 4-j$;
 (ii)  $S'\supseteq F\cup \{y\}$; (iii)  $S'$ misses at least $j$
 vertices in $N$ and the neighbors of these vertices distinct from $y$. 
By construction, the $S'$'s obtained above are  in $\mathcal{S}$.

Let $A\in \mathcal{S}$ contain fewest neighbors of $y$. If $A\cap N = \emptyset$, then $\rho_{G,\bc}(A-y)\leq 4-j-5+3 = 2-j$, a contradiction. Thus we may assume $u_k, u_k'\notin A$ for $k\in [j]$, and $u_l, u_l'\in A$ for some $l>j$. 
By choosing $N_2' = \{u_1', u_2',\dots, u_{j-1}', u_l'\}$, we can find a set $B $ that is also in $\mathcal{S}$, where $u_k, u_k'\notin B$ for $k\in [j-1]$ and also $u_l, u_l'\notin B$. 

Then $u_k, u_k'\notin A\cup B$ for $k\in [j-1]$, hence $A\cup B$ is nontrivial. By the maximality of $|F|$, $\rho_\bc(A\cup B)\geq 4-j$. By submodularity,
$\rho(A\cap B)\leq \rho(A\cup B)+\rho(A\cap B)\leq \rho(A)+\rho(B)\leq 2(4-j) \leq 4-j$. Thus $A\cap B$ is also a member of $\mathcal{S}$. However, $u_l\notin A\cap B$, a contradiction to the choice of $A$.
\end{proof}

\begin{lemma}\label{24geq1}
Suppose $\emptyset\neq F \subsetneq V(G)$ is a nontrivial set, and $\rho_{G,\bc}(F) \leq 4-j$. Then $F$ is obtained from $V(G)$ by deleting two surplus vertices.
\end{lemma}
\begin{proof}
Let $F$ be a counterexample to the lemma maximum in size.
By Lemma~\ref{nsnbg}, $F$ has an ordinary neighbor $y$. Let $x$ be the neighbor of $y$ in $F$, where $L(x) = \{x_1, x_2\}, L(y) = \{y_1, y_2\}$, and $x_1\sim y_1, x_2\sim y_2$.
Fix a $\bc$-coloring $\varphi$ of $G[F]$ with $\varphi(x) = x_1$. 

We construct $G',H'$ and $\bc'$ as follows.
Replace $F$ in $G$ by a single vertex $v$, where $L(v) = \{v_1, v_2\}$.  For every vertex $u\in N_G(F)-F-y$, denote $L(u) = \{u_1, u_2\}$ so that the neighbor of $u_1$ in $H$ is colored by $\varphi$. In $G'$, add an edge between $v$ and each vertex in $N_G(F)-y$.
In $H'$, let $v_1u_1, v_2u_2\in E(H')$. 
Let $\bc'(v_1) = 0, \bc'(v_2) = -1, \bc'(y_1) = \bc(y_1)-1$, and let $\bc'$ coincide with $\bc$ for all other nodes of $H'$.

If $G'$ has a $\bc'$-coloring  $\psi$, then $\psi(v) = v_1$ and $\varphi\cup \psi$ is a $\bc$-coloring of $G$, a contradiction.
Thus $G'$ has no $\bc'$-coloring, and by the minimality of $G$,
 there is a set $S\subset V(G')$ with $\rho_{G'\bc'}(S)\leq 1-j$.

Let $S$ be such set maximal in size. Then $v\in S$. 
If $y\notin S$, then $\rho_{G,\bc}((S-v)\cup F)\leq 1-j+2 = 3-j$. Since 
$y\notin (S-v)\cup F$, this  contradicts Lemma~\ref{geq0for24}. If $y\in S$, $\rho_{G,\bc}((S-v)\cup F)\leq 1-j+2+1-3 = 1-j$, a contradiction again.
\end{proof}

\begin{corollary}\label{924}
Suppose all neighbors of a vertex $w\in V(G)$ are surplus vertices.
If $r\geq 3$ and $w$ is an $(m;2,r)$-vertex, then $m\geq 8$ and if
$r\geq 4$ and  $w$ is an $(m;2,r)$-vertex, then $m\geq 10$.
\end{corollary}
\begin{proof}
Suppose first that $r\geq 3$ and $w$ is an $(m;2,r)$-vertex where $m\leq 4+r$.
Denote $N = N(w) = \{u_1, \dots, u_m \}$, and for $1\leq h\leq m$,
let $u'_h$ be the neighbor of $u_h$ distinct from $w$. Let
$N' = \{ u'_1, \dots, u'_m \}$.

Let $\bc'$ differ from  $\bc$ only in that $\bc'(u'_1) = \bc(u'_1)-(1,1)$. Then by Lemma~\ref{24geq1} and the minimality of $G$, $G-w-N$ has a $\bc'$-coloring $\varphi$. 
We extend $\varphi$ to $G$: let $\varphi(u_k)\nsim \varphi(u'_k)$ for each $2\leq k\leq m$. Since $m-1< (2+1)+(r+1)$, there is a node $\alpha(w)\in L(w)$, such that the number of its already colored neighbors is at most its capacity. Color $w$ by $\alpha(w)$.
Extend $\varphi$ to $u_1$ so that $\varphi(u_1)\nsim \alpha(w)$.
Then $\varphi$ is a $\bc$-coloring on $G$, a contradiction.

The only case not covered by the argument above is that $r=4$ and $m=9$.
Let $S\subset V(G-w-N)$, and $N_S \subset N$ be the set of surplus vertices connecting $S$ and $w$. 
Suppose there are $h,h'\in [9]$, such that every $S\subset V(G-w-N)$ containing $u'_h$ and $u'_{h'}$ has potential at least two. 
Let $\bc'$ differ from  $\bc$ only in that  each node  $\beta\in L(u'_h)\cup L(u'_{h'})$ has  $\bc'(\beta)=\bc(\beta)-1$. 
By the choice of $h,h'$, $G-w-N$ has a $\bc'$-coloring $\varphi$. We extend $\varphi$ to $N\setminus \{u_h, u_{h'}\}$, so that $\varphi(u_k)\nsim \varphi(u'_k)$ for each $k\in [9]\setminus\{h,{h'}\}$.
Then there is a node $\alpha(w)\in L(w)$, such that the number of its already colored neighbors is at most its capacity. Color $w$ by $\alpha(w)$.
Extend $\varphi$ to $u_h, u_{h'}$, so that $\varphi(u_h), \varphi(u_{h'})\nsim \alpha(w)$.
Then $\varphi$ is a $\bc$-coloring on $G$, a contradiction.
\\
Thus we may assume that for each pair of $h,h'\in [9]$, there is some $S\subset V(G-w-N)$ containing $u_h, u_{h'}$  with potential at most one.
Let $\mathcal{F}$ be the family of all subsets of $V(G-w-N)$ whose potential is at most one.
Take any $M\in \mathcal{F}$, and let $N_M$ denote the set of surplus vertices connecting $M$ and $w$.
If $|N_M| = 9$, then 
$\rho_{G,\bc}(M+w+N_M)\leq \rho_{G,\bc}(M)+5-9\leq -3$, a contradiction.
If $|N_M|\geq 6$, then 
$\rho_{G,\bc}(M+w+N_M)\leq \rho_{G,\bc}(M)+5-6\leq 0$, a contradiction to Lemma~\ref{24geq1}, since $N'\setminus (M+w+N_M)\neq \emptyset$.
Let $M\in \mathcal{F}$ with $|N_M|$ maximum. We may assume $u'_1\in M, u_2\notin M$.
Then there is some $M'\in \mathcal{F}$ containing $u'_1, u'_2$.
By submodularity,
\begin{equation}\label{MM'}
\rho_{G,\bc}(M\cap M')+\rho_{G,\bc}(M\cup M')\leq \rho_{G,\bc}(M)+\rho_{G,\bc}(M')\leq 1+1=2
\end{equation}
By Lemma~\ref{24geq1}, $\rho_{G,\bc}(M\cap M')\geq 1$, so by \eqref{MM'},  $\rho_{G,\bc}(M\cup M')\leq 1$. But then $|N_{M\cup M'}|> |N_M|$, a contradiction to the choice of $M$.
\end{proof}

\begin{lemma}\label{24chleq0}
$ch(v)\leq 0$ for all $v\in V(G)$.
\end{lemma}
\begin{proof}
Suppose $ch(v)>0$ for some $v\in V(G)$. Then we have
\begin{equation}\label{charg}
 ch(v)=c_1+c_2+3-j-\frac{3}{2}d_1 -\frac{1}{2}d_2 \geq \frac{1}{2}.
\end{equation}
Let $N_1$ denote the set of ordinary neighbors of $v$ and $N_2$ the set of surplus neighbors of $v$. For $u\in N_2,$ let $g(u)$ denote the  neighbor of  $u$ distinct from $v$.
Let   $N'_2=\{g(u): u\in N_2\}$.

For $u\in N_1$, a node 
$u_\alpha\in L(u)$ is \emph{conflicting} with $v_\beta\in L(v)$ if  $u_\alpha\sim v_\beta$ in $H$. For $u\in N_2$, a node 
$g(u)_\alpha\in L(u)$ is \emph{$u$-conflicting} with $v_\beta\in L(v)$ if  
the neighbors of $g(u)_\alpha$ and $v_\beta$ in $L(u)$ are distinct.

Vertices $x,y\in N_2$ with  $g(x)=g(y)=u\in N'_2$ are \emph{parallel}, if the $x$-conflicting node in $L(u)$ is also $y$-conflicting.
 Otherwise, we call $(x,y)$ a \emph{twisted} pair.

We aim to construct an auxiliary graph and either find a $\bc$-coloring of this graph such that $v$ is colored by $r(v)$ and extend this coloring to a $\bc$-coloring of $G$, or find a low-potential set in the new graph, which leads to a contradiction to the choice of $(G,\bc)$.
\vspace{2mm}

Construct $G', \bc', H'$ from $G,\bc, H$ as follows:

Step 0: Initialize $G' = G, \bc' = \bc, H' = H$.

Step 1: For each $u\in N_1$, let $u_\alpha\in L(u)$ be conflicting with $r(v)$. Remove edge $uv$ from $G'$, remove edges between $L(u)$ and $L(v)$ from $H'$, and decrease the capacity of $u_\alpha$ by one in $\bc'$.

Step 2: For each $u'\in N'_2$, if there are $u_1,u_2\in N_2$ connecting $u'$ and $v$ that form a twisted pair, then remove $u_1,u_2$ from $G'$, and remove $L(u_1), L(u_2)$ from $H'$. Repeat this step until there are no such $u',u_1,u_2$.

Step 3: For each $u'\in N'_2$, if there are $u_1,u_2\in N_2$ connecting $u'$ and $v$ that form a parallel pair, then let $u'_\alpha\in L(u')$ be $u_1$-conflicting with $r(v)$.
Remove $u_1,u_2$ from $G'$, and remove $L(u_1), L(u_2)$ from $H'$. Decrease the capacity of $u'_\alpha$ by one in $\bc'$. Repeat this step until there are no such $u',u_1,u_2$.

At this point, each vertex in $N'_2$ has at most one common neighbor with 
 $v$. Denote the set consisting of vertices in $N'_2$ having now exactly one neighbor in $N_2$ by $N''_2$. 

Step 4.1:
If $|N''_2|$ is odd, then take any $u'_0\in N''_2$, let $u_0$ be the surplus vertex connecting $u'_0$ and $v$. Delete $u_0$ from $G'$, and delete $L(u_0)$ from $H'$.

Step 4.2:
Now we may assume that $|N''_2|$ is even. 
Take $w',u'\in N''_2$, let $w'_\alpha\in L(w')$ and $u'_\alpha\in L(u')$ be conflicting with $r(v)$.
Delete the surplus neighbors $u,w$ of $v$ adjacent to $w',u'$ from $G'$, and delete their lists from $H'$.
Add a surplus vertex $z$ to $G'$ adjacent to $w'$ and $u'$. And in $H'$, let $L(z) = \{z_1, z_2\}$, such that $z_1\sim w'_\alpha, z_2\sim u'_\alpha$.
Repeat the above until $N_2$ is empty. Denote the set consisting of all newly added surplus vertices $z$ by $N_3$. Then $|N_3|\leq \lfloor \frac{|N'_2|}{2}\rfloor$.

Step 5: At this point, $v$ is an isolated vertex in $G'$. Delete $v$ from $G'$, and delete $L(v)$ from $H'$. The resulting $G',H'$ and $\bc'$ are final.
\vspace{2mm}

Suppose $G'$ has a $(\bc', H')$-coloring $\phi$. We now extend $\phi$ to 
the deleted vertices of $G$
following the steps of constructing $H'$ in the reversed order.

Step $5^-$: Let $\phi(v)=r(v)$. Then $\phi(v)$ has at most 
 $|N_1|$ neighbors in $H_\phi$ and 
 each $u\in N_1$ gets at most one extra conflicting neighbor.

Step $4.2^-$: For each $z\in N_3$ and its neighbors $w',u'\in N''_2$,
 let $w,u$ be the surplus vertices in $G$ connecting  $v$ with $w'$ and $u'$, respectively at the beginning of Step 4.2. We will delete $z$ and assign colors to $w$ and $u$ as follows.
 If $\phi(z)\sim \phi(u')$, then we choose $\phi(u)\nsim r(v)$ and $\phi(w)\nsim \phi(w')$. In this way, the degrees of $\phi(u')$ and $ \phi(w')$ in $H_\phi$ do not increase, and the degree of $r(v)$ increases by at most $1$. If $\phi(z)\nsim \phi(u')$ but $\phi(z)\sim \phi(w')$, then we switch the roles of $u$ and $w$.
 If $\phi(z)\nsim \phi(u')$  and $\phi(z)\nsim \phi(w')$, then by Step 4.2,
 either $\phi(u')$ is not $u$-conflicting with $r(v)$, or
 $\phi(w')$ is not $w$-conflicting with $r(v)$. 
Then after we choose $\phi(u)\nsim \phi(u')$ and $\phi(w)\nsim \phi(w')$, again the degree of $r(v)$ in $H_\phi$ increases only by $1$ and the degrees of $\phi(u')$ and $ \phi(w')$ in $H_\phi$ do not increase.

Step $4.1^-$: If $|N''_2|$ is odd, then we let  $\phi(u_0)\nsim \phi(u'_0)$.
So the degree of $r(v)$ in $H_\phi$ increases  by at most $1$ and the degree of $\phi(u'_0)$  does not change.

Step $3^-$:  For every $u'\in N'_2$ and for each  parallel pair $u_1,u_2\in N_2$ connecting $u'$ and $v$ deleted on Step 3, we let  $\phi(u_1)\nsim \phi(u')$ and $\phi(u_2)\nsim \phi(r(v))$. So, the degree of $r(v)$ in $H_\phi$ increases  by at most $1$ and the degree of $\phi(u')$  increases  by at most $1$.

Step $2^-$:  For every $u'\in N'_2$ and and for each twisted pair $u_1,u_2\in N_2$ connecting $u'$ and $v$ deleted on Step 2, we let  $\phi(u_1)\nsim \phi(u')$ and $\phi(u_2)\nsim \phi(u')$. Then the  degree of $\phi(u')$ does not increase and since $u_1,u_2$ is a twisted pair,
 the degree of $r(v)$ in $H_\phi$ increases  by at most $1$.




\vspace{2mm}

Now by construction, the degree in $H_\phi$ of every $u_\alpha$ apart from possibly  $r(v)$ is at most its capacity. 
Let $c_1 = \bc_1(v), c_2 = \bc_2(v), d_1 = |N_1|, d_2 = |N_2|$.
By construction,  $r(v)$ has at most $|N_1|+ \lceil |N_2|/2\rceil=d_1+\lceil d_2/2\rceil$ colored neighbors. Thus, if
\begin{equation}\label{suff}
   c_2= \bc(r(v))\geq d_1+\lceil d_2/2\rceil,
\end{equation}
then $\phi$ is a $\bc$-coloring of $G$.
By (\ref{charg}) and $j\geq 4$,
$$
d_1+\frac{1}{2}d_2\leq \frac{3}{2}d_1+\frac{1}{2} d_2 \leq c_1+c_2+\frac{5}{2}-j \leq c_2+\frac{1}{2}.
$$
This implies that if $d_2$ is even, or $d_1$ is positive, or $c_1\leq 1$, or $j\geq 5$, then~\eqref{suff} holds. So, assume that $d_2$ is odd, $d_1=0$, $c_1=2$ and $j=4$.
This means
 $\bc(v) = (2,c_2)$, $d(v) = 2c_2+1$, $c_2\leq j=4$ and all neighbors of $v$ are surplus vertices.
If $c_1+c_2+1 = c_2+3\geq d(v)=2c_2+1$, then we can extend any $\bc$-coloring of $G-v$ to $v$ greedily. The only remaining cases are $c_2 = 3, d(v) = d_2 = 7$ and  $c_2 = 4, d(v) = d_2 = 9$.
 By Corollary~\ref{924}, such $v$ does not exist. Thus $\phi$ is a $\bc$-coloring of $G$, a contradiction.

\bigskip
Therefore, we may assume that $G'$ is not $\bc'$-colorable.
By the minimality of $G$, there is an $S\subset V(G')$ with $\rho_{G',{\bf c'}}(S) \leq 1-j$. Let $S$ be such a set of minimum potential and modulo this maximal in size. 

Let $S_1=S\cap N_1$, $s_2=\sum_{w\in S\cap N'_2}(\bc(w)-\bc'(w))$,
$S_3=S\cap N_3$, and $S'=S-S_3$. Then $S'\subset V(G)$ and $v\notin S'$.
So by Lemma~\ref{24geq1}, $ \rho_{G,{\bf c}}(S')\geq 5-j$.

Recall that while constructing $\bc'$, if we decreased the capacity of a vertex, then either this vertex is in $N_1$ (in Step 1) or has two common neighbors with $v$ (in Step 3).  It follows that
\begin{equation}\label{s'}
    5-j\leq\rho_{G,{\bf c}}(S')=\rho_{G',{\bf c'}}(S)+|S_3|+|S_1|+s_2\leq
    1-j+|S_3|+|S_1|+s_2.
\end{equation}

Let $N_4\subset N_2$ be the set of surplus vertices  in $G$ connecting $v$ and $S\cap N'_2$, and $V':= (S'\cup \{v\}\cup N_4) 
$. Recall that 
 if $S$ contains a  $z\in N_3$, then it also contains both its neighbors $u',w',$ each of which has a common neighbor  with $v$ in $G$ (that belongs to $N_4$) and is adjacent to only one vertex in $N_3$ in $G'$. Also, 
  while constructing $\bc'$,
 every time when we decreased the capacity of a vertex $u'\in N'_2$ in Step 3, we have deleted two its common neighbors with $v$, and these vertices are now in $N_4$. It follows that $|N_4|\geq 2|S_3|+2s_2$.
By this and~\eqref{s'},
\\
\begin{equation}\label{replace}
     \rho_{G,{\bf c}}(V') =  \rho_{G,{\bf c}}(S')+\rho_\bc(v) -3|S_1|-|N_4|
  \leq  \rho_{G',{\bf c'}}(S) +\rho_\bc(v) -2|S_1|-\left\lceil\frac{|N_4|}{2}\right\rceil
\end{equation}
 $$  
    \leq 1-j+\rho_\bc(v)-2|S_1|-\left\lceil\frac{|N_4|}{2}\right\rceil.$$
Since $S$ is maximal in size, we may assume that if $V'\neq V(G)$, then there is some ordinary vertex in $V(G)\setminus V'$, otherwise we can just include the surplus vertices outside of $S$ into $S$ to make its  size larger. Thus by Lemma~\ref{24geq1}, 
\begin{equation}\label{5-j}
    \mbox{$\rho_{G,{\bf c}}(V')\geq 5-j$ when $V'\neq V(G)$.}
\end{equation}
So, if  $V'\neq V(G)$, then~\eqref{replace} and~\eqref{5-j} together yield
$$4+2|S_1|+\left\lceil\frac{|N_4|}{2}\right\rceil\leq \rho_\bc(v)\leq 5.$$
This implies $S_1=\emptyset$  and  $|N_4|\leq 2$. 
But then $\rho_{G,\bc}(S') \leq \rho_{G',\bc'}(S) + 1\leq 2-j$, a contradiction to~\eqref{s'}. Therefore, $V'=V(G)$, which means
$S_1=N_1$
and $N_4=N_2$.


Since~\eqref{s'} yields $4\leq |S_3|+|S_1|+s_2$, we infer from
$|N_4|\geq 2|S_3|+2s_2$ that 
\begin{equation}\label{s3}
   |N_1|+\left\lfloor\frac{|N_2|}{2}\right\rfloor\geq |S_3|+|S_1|+s_2\geq 4. 
\end{equation}
Then, since $\rho_{G,{\bf c}}(V')\geq 2-j$,~\eqref{replace} gives
$$2-j\leq 1-j+\rho_\bc(v)-2|N_1|-\left\lceil\frac{|N_2|}{2}\right\rceil\leq
1-j+\rho_\bc(v)-|N_1|-4.
$$
For this to happen, we need $\rho_\bc(v)=5$, $N_1=\emptyset$ and
$|N_2|$ be even. Now,~\eqref{s3} yields $|N_2|\geq 8$ and~\eqref{charg}
yields $|N_2|\leq 9$. Since $|N_2|$ is even, we have $|N_2|= 8$.
By Corollary~\ref{924}, $G$ has no such vertices. \end{proof}

Now Lemma~\ref{crutial} together with Lemma~\ref{24chleq0} complete the proof of
Theorem~\ref{MT-F}.

\section{A construction}

A construction in~\cite{JKMX} shows that the bounds in Theorem~\ref{MT-f}
are sharp for each  pair $(i,j)$ satisfying the theorem for infinitely many $n$. For the convenience of the reader, we repeat this construction below, but do not present the proof of its properties. The interested readers may find it in Section 5  of \cite{JKMX}.

\medskip
 Fix $i\in \{1,2\}$ and $j\geq 2i$.

\begin{definition}
    Given a vertex $v$ in a graph $G$, a {\em flag at $v$} is a subgraph $F$ of $G$ with $i+3$ vertices $v, x, u_1,\dots,u_{i+1}$, such that $d_G(x)=i+2$,   $vx\in E(F)$, and  $u_1,\ldots,u_{i+1}$ are $2$-vertices adjacent to both $v$ and $x$. We call $v$ the {\em base vertex}, $x$ the {\em top vertex}, and $u_1,\dots, u_{i+1}$ the {\em middle vertices} 
     of the flag $F$.
\end{definition}
\begin{figure}[h]
    \centering
    \includegraphics[width=1.7in]{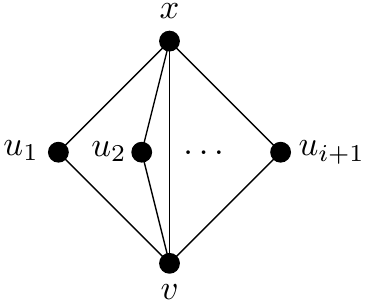}
\caption{A flag at base vertex $v$.}\label{fig:flag}
\end{figure}
 We call a vertex being the base of $k$ flags a {\em $k$-base} vertex. A graph with a $k$-base vertex $v$ (and all the flags based at $v$) contains exactly $1+k(i+2)$ vertices and $k(1+2(i+1))$ edges.

We now define our critical graph $G_m$ for given positive integer $m$: when $m = 1$, let $v$ be an $(i+j+2)$-base vertex; when $m\geq 2$, let $v_1,\dots, v_m$ be a path, where $v_1$ is an $(i+1)$-base vertex, $v_m$ is an $(i+j+1)$-base vertex, and $v_k$ is an $i$-base vertex for all $k = 2,\dots, m-1$.
One can easily check that  $|V(G_m)|=(i+2)(mi+j+2)+m$ and $|E(G_m)|=(2i+3)(mi+j+2)+m-1$, thus
\[
|E(G_m)|=\frac{(2i+1)|V(G_m)|+j-i+1}{i+1}.
\]

For the cover graph $H_m$ of $G_m$, we need the following definition:
\begin{definition}
    A flag with base vertex $v$, top vertex $x$, and middle vertices $u_1,\dots, u_{i+1}$ is called {\em parallel}, if $p(u)\sim p(w)$, $r(u)\sim r(w)$ for each edge $uw$ in the flag;
    the flag is called {\em twisted} if $p(v)\sim r(x)$, $p(x)\sim r(v)$, $p(x)\sim p(u_k), r(x)\sim r(u_k)$, and $p(v)\sim r(u_k), p(u_k)\sim r(v)$ for each $k\in [i+1]$.
\end{definition}

To construct a cover graph $H_m$ that does not admit an $(i,j)$-coloring, we let all $i+1$ flags at $v_1$ and all $i$ flags at $v_k$ for $k=2,\dots, m-1$ be twisted. Among the $i+j+1$ flags at $v_m$, let $i+1$ of them be twisted and the remaining $j$ be parallel. For edge $v_k v_{k+1}$ on the path, let $p(v_k)\sim r(v_{k+1}), r(v_k)\sim p(v_{k+1})$ for each $k = 1,\dots, m-1$.

\end{document}